\documentclass[a4paper,12pt,reqno]{amsart}
\usepackage[text={155mm,220mm}]{geometry}
\usepackage{textcomp}
\theoremstyle{plain}
\newtheorem{theorem}{Theorem}
\newtheorem*{theorem*}{Theorem}

\newtheorem*{corollary*}{Corollary}
\newtheorem{lemma}{Lemma}
\newtheorem*{lemma*}{Lemma}

\newtheorem*{proposition*}{Proposition}

\newtheorem*{conjecture*}{Conjecture}
\theoremstyle{definition}

\newtheorem*{definition*}{Definition}
\theoremstyle{remark}

\newtheorem*{remark*}{Remark}

\begin{document}

\title[The generalized shift operator, Cantor series, and rational numbers]{A note on rational numbers and certain operators (The generalized shifts and rational numbers)}
\author{Symon Serbenyuk}
\address{
  45~Shchukina St. \\
  Vinnytsia \\
  21012 \\
  Ukraine}
\email{simon6@ukr.net}



\begin{abstract} This paper is devoted to conditions defined  in terms of the generalized shift operator for a rational number to be representable by certain positive generalizations of $q$-ary expansions.
\end{abstract}

\maketitle

\maketitle



\section{Introduction}

The problem on conditions for a rational number to be representable by the following positive series was introduced by Georg Cantor in the paper \cite{Cantor1} in 1869: 
\begin{equation}
\label{eq:  series 1}
\frac{\varepsilon_1}{q_1}+\frac{\varepsilon_2}{q_1q_2}+\dots +\frac{\varepsilon_k}{q_1q_2\dots q_k}+\dots ,
\end{equation}
where $Q\equiv (q_k)$ is a fixed sequence of positive integers, $q_k>1$,  and $(\Theta_k)$ is a sequence of the sets $\Theta_k\equiv\{0,1,\dots ,q_k-1\}$,  as well as $\varepsilon_k\in\Theta_k$.

Series of form \eqref{eq:  series 1} are called \emph{Cantor series}. By $\Delta^Q _{\varepsilon_1\varepsilon_2...\varepsilon_k...}$  denote any number $x\in [0,1]$  having expansion \eqref{eq:  series 1}. This notation is called \emph{the representation of  $x\in [0,1]$ by Cantor series \eqref{eq:  series 1}}. 

It is easy to see that  Cantor series expansion  \eqref{eq:  series 1} is the $q$-ary expansion 
$$
\frac{\varepsilon_1}{q}+\frac{\varepsilon_2}{q^2}+\dots+\frac{\varepsilon_k}{q^k}+\dots
$$
of real numbers from~$[0,1]$, where $\varepsilon_k\in\{0,1, \dots , q-1\}$,  whenever the condition $q_k=const=q$ holds for all $k\in\mathbb N$ ($\mathbb N$  is the set of all positive integers),  where $1<q\in\mathbb N$. 

A number of researches are devoted to investigations of Cantor expansions from different points of view (a brief description is given in \cite{Serbenyuk17}) including studying the problem on Cantor series expansions of rational numbers (for example, see \cite{{Cantor1},  {Diananda_Oppenheim1955}, {Hancl_Tijdeman2004}, S13,  {Serbenyuk17}, Serbenyuk2017,  Serbenyuk20, Symon2021,
 Tijdeman_Pingzhi2002}). In~\cite{Galambos1976},    Prof.~J\'anos~Galambos calls the problem on representations of rational numbers by Cantor series \eqref{eq:  series 1} as  \emph{the fourth open problem}. 

One can note that the notion of the shift operator is applicable to this problem (for example, some descritions are given in \cite{{Serbenyuk17}, Symon2021}). This paper is devoted to applications of the notion of the generalized shift operator to solving  the problem on representations of rational numbers by positive Cantor series. The present research is the continuation of investigations presented in the papers~\cite{Serbenyuk20, Symon2021}.

\section{The shift and generalized shift operators}

 \emph{The shift operator $\sigma$ of expansion \eqref{eq:  series 1}} is a map of the following form
$$
\sigma(x)=\sigma\left(\Delta^Q _{\varepsilon_1\varepsilon_2\ldots \varepsilon_k\ldots}\right)=\sum^{\infty} _{k=2}{\frac{\varepsilon_k}{q_2q_3\dots q_k}}=q_1\Delta^{Q} _{0\varepsilon_2\ldots \varepsilon_k\ldots}.
$$
It is easy to see that 
\begin{equation*}
\begin{split}
\sigma^n(x) &=\sigma^n\left(\Delta^Q _{\varepsilon_1\varepsilon_2\ldots \varepsilon_k\ldots}\right)\\
& =\sum^{\infty} _{k=n+1}{\frac{\varepsilon_k}{q_{n+1}q_{n+2}\dots q_k}}=q_1\dots q_n\Delta^{Q} _{\underbrace{0\ldots 0}_{n}\varepsilon_{n+1}\varepsilon_{n+2}\ldots}.
\end{split}
\end{equation*}

One can note the partial case of this operator  for  $q$-ary expansions
$$
\sigma^n\left(\Delta^q _{\varepsilon_1\varepsilon_2\ldots \varepsilon_k\ldots}\right) =\sum^{\infty} _{k=n+1}{\frac{\varepsilon_k}{q^{k-n}}}=\Delta^{q} _{\varepsilon_{n+1}\varepsilon_{n+2}\ldots}.
$$

Suppose a number $x\in [0,1]$ is  represented by  series \eqref{eq: series 1}. Then  \emph{the generalized shift operator $\sigma_m$} is a map of the following form:
$$
\sigma_m(x)=\sigma_m\left(\Delta^Q _{\varepsilon_1\varepsilon_2\ldots \varepsilon_k\ldots}\right)=\sum^{m-1} _{k=1}{\frac{\varepsilon_k}{q_1q_2\cdots q_k}}+\sum^{\infty} _{t=m+1}{\frac{\varepsilon_t}{q_1q_2\cdots q_{m-1}q_{m+1}\cdots q_t}}.
$$
That is,  any number from $[0,1]$ can be represented by two fixed sequences $(q_k)$ and $(\varepsilon_k)$ (Cantor series expansions). The generalized shift operator maps the preimage into a number represented by the following two sequences $(q_1, q_2, \dots , q_{m-1}, q_{m+1}, q_{m+2}, \dots )$ and $(\varepsilon_1, \varepsilon_2, \dots , \varepsilon_{m-1}, \varepsilon_{m+1}, \varepsilon_{m+2}, \dots )$.  Properties of this operator are considered more detail in \cite{{Symon2021-2}}.

Denote by  $\vartheta_{m}$ the sum $\sum^{m} _{k=1}{\frac{\varepsilon_k}{q_1q_2\cdots q_k}}$ and by $\delta_m$ the sum
$$
\varepsilon_1q_2q_3\cdots q_m+\varepsilon_2q_3q_4\cdots q_m+\dots+\varepsilon_{m-1}q_m
+\varepsilon_m.
$$
 Then
\begin{equation}
\label{eq: generalized shift 1}
\sigma_m(x)=q_mx-(q_m-1)\vartheta_{m-1}-\frac{\varepsilon_m}{q_1q_2\cdots q_{m-1}}.
\end{equation}

For the case of positive Cantor series, the notion of the generalized shift operator is considered more detail in~\cite{Symon2021-2} (see also~\cite{Serbenyuk2017}, where the shift and generalized shift operators are considered for the case of alternating Cantor series). 

Let us remark that the following statement is true.

\begin{lemma}
For the generalized shift operator defined in terms of positive Cantor series, the following relationships hold:
\begin{itemize}
\item
$$
\sigma_{m+1}(x)=\frac{q_{m+1}}{q_m}\sigma_m(x)-\frac{q_{m+1}-q_m}{q_m}\vartheta_{m-1}-\frac{\varepsilon_{m+1}-\varepsilon_m}{q_1q_2\cdots q_{m-1}q_m};
$$
\item
$$
\varepsilon_m=q_1q_2\cdots q_mx-q_1q_2\cdots q_{m-1}\sigma_m(x)-(q_m-1)\delta_{m-1}.
$$
\end{itemize}
\end{lemma}
\begin{proof}
Let us prove \emph{the first relationship}. Using \eqref{eq: generalized shift 1}, we get
$$
\frac{\sigma_m(x)}{q_m}+\frac{q_m-1}{q_m}\vartheta_{m-1}+\frac{\varepsilon_m}{q_1q_2\cdots q_{m-1}q_m}=x=\frac{\sigma_{m+1}(x)}{q_{m+1}}+\frac{q_{m+1}-1}{q_{m+1}}\vartheta_{m}+\frac{\varepsilon_{m+1}}{q_1q_2\cdots q_mq_{m+1}},
$$ 
$$
{q_{m+1}\sigma_m(x)}+{q_{m+1}(q_m-1)}\vartheta_{m-1}+\frac{q_{m+1}\varepsilon_m}{q_1q_2\cdots q_{m-1}}={q_m\sigma_{m+1}(x)}+q_m(q_{m+1}-1)\vartheta_{m}+\frac{\varepsilon_{m+1}}{q_1q_2\cdots q_{m-1}}.
$$ 

Since $\vartheta_m=\vartheta_{m-1}+\frac{\varepsilon_m}{q_1q_2\cdots q_m}$, we obtain
$$
q_{m+1}\sigma_m(x)=q_{m+1}\vartheta_{m-1}+q_m\sigma_{m+1}(x)-q_m\vartheta_{m-1}-\frac{\varepsilon_m}{q_1q_2\cdots q_{m-1}}+\frac{\varepsilon_{m+1}}{q_1q_2\cdots q_{m-1}}.
$$
Hence
$$
\sigma_{m+1}(x)=\frac{q_{m+1}}{q_m}\sigma_m(x)-\frac{q_{m+1}-q_m}{q_m}\vartheta_{m-1}-\frac{\varepsilon_{m+1}-\varepsilon_m}{q_1q_2\cdots q_{m-1}q_m}.
$$

Let us prove \emph{the second relationship}. Using \eqref{eq: generalized shift 1}, we have
$$
q_1q_2\cdots q_{m-1}\sigma_m(x)=q_1q_2\cdots q_mx-q_1q_2\cdots q_{m-1}(q_m-1)\vartheta_{m-1}-\varepsilon_m.
$$
The relationship follows from the last-mentioned equality.
\end{proof}

\section{Rational numbers}

\begin{theorem}
A number $x\in [0,1]$ represented by series \eqref{eq: series 1} is a  rational number  if and only if  there exist non-negative integers $m_1$ and $m_2$ such that $m_1\ne m_2$ and the condition
$$
\left\{q_1q_2\cdots q_{m_1-1}\sigma_{m_1}(x)\right\}=\left\{q_1q_2\cdots q_{m_2-1}\sigma_{m_2}(x)\right\}
$$
holds, where $\{a\}$ is the fractional part of $a$, $q_{-1}=q_0=1$, and $\sigma_0(x)=x$.
\end{theorem}
\begin{proof} Let us prove that  \emph{the necessity} is true.  Let $x$ be a rational number, i.e., $x=\frac a b$, where $a\in\mathbb Z_0=\mathbb N\cup \{0\}$ and $b\in\mathbb N$, $a<b$, and $(a,b)=1$.

Let us consider the  sequence $(q_1q_2\cdots q_{k-1}\sigma_{k}(x))$. Using \eqref{eq: generalized shift 1}, we have
$$
q_1q_2\cdots q_{k-1}\sigma_{k}(x)=q_1q_2\cdots q_{k-1}q_kx-\varepsilon_{k}-(q_k-1)(\varepsilon_1q_2q_3\cdots q_{k-1}+\dots + \varepsilon_{k-2}q_{k-1}+\varepsilon_{k-1})
$$
$$
=q_1q_2\cdots q_{k}\frac a b+(\varepsilon_1q_2q_3\cdots q_{k-1}+\dots + \varepsilon_{k-2}q_{k-1}+\varepsilon_{k-1})-(\varepsilon_1q_2q_3\cdots q_{k}+\dots + \varepsilon_{k-1}q_{k}+\varepsilon_{k})
$$
$$
=\frac{aq_1q_2\cdots q_k+b\delta_{k-1}-b\delta_k}{b}.
$$

Since in our case 
$$
\sigma^k(x)=q_1q_2\cdots q_kx-(\varepsilon_1q_2q_3\cdots q_{k}+\dots + \varepsilon_{k-1}q_{k}+\varepsilon_{k})=\frac{q_1q_2\cdots q_ka-b\delta_k}{b},
$$
we obtain
$$
q_1q_2\cdots q_{k-1}\sigma_{k}(x)=\delta_{k-1}+\sigma^k(x)=\delta_{k-1}+\frac{q_1q_2\cdots q_ka-b\delta_k}{b}.
$$

It is easy to see that 
$$
\{\sigma^k(x)\}=\begin{cases}
\sigma^k(x)-1&\text{whenever $x=\Delta^Q _{\varepsilon_1\varepsilon_2...\varepsilon_{k-1}[\varepsilon_k-1][q_{k+1}-1][q_{k+2}-1]...}$}\\
\sigma^k(x)&\text{whenever $x\ne\Delta^Q _{\varepsilon_1\varepsilon_2...\varepsilon_{k-1}[\varepsilon_k-1][q_{k+1}-1][q_{k+2}-1]...}$.}
\end{cases}
$$
Hence if $x\ne\Delta^Q _{\varepsilon_1\varepsilon_2...\varepsilon_{k-1}[\varepsilon_k-1][q_{k+1}-1][q_{k+2}-1]...}$, then
$$
\left[q_1q_2\cdots q_{k-1}\sigma_{k}(x)\right]=\delta_{k-1}=\varepsilon_1q_2q_3\cdots q_{k-1}+\varepsilon_2q_3q_4\cdots q_{k-1}+\dots +\varepsilon_{k-2}q_{k-1}+\varepsilon_{k-1}
$$
and
$$
\left\{q_1q_2\cdots q_{k-1}\sigma_{k}(x)\right\}=\sigma^k(x).
$$
If $x=\Delta^Q _{\varepsilon_1\varepsilon_2...\varepsilon_{k-1}[\varepsilon_k-1][q_{k+1}-1][q_{k+2}-1]...}$, then 
$$
\left[q_1q_2\cdots q_{k-1}\sigma_{k}(x)\right]=1+\delta_{k-1}=1+\varepsilon_1q_2q_3\cdots q_{k-1}+\varepsilon_2q_3q_4\cdots q_{k-1}+\dots +\varepsilon_{k-2}q_{k-1}+\varepsilon_{k-1}
$$
and
$$
\left\{q_1q_2\cdots q_{k-1}\sigma_{k}(x)\right\}=0.
$$
Here $[x]$ is the integer part of $x$ and $\{x\}$ is the fractional part of $x$.

By analogy to arguments described in~\cite{Symon2021}, we get
$$
\left\{q_1q_2\cdots q_{k-1}\sigma_{k}(x)\right\}=\left\{\frac{q_1q_2\cdots q_ka-\delta_kb}{b}\right\}=\left\{\frac{a_k}{b}\right\},
$$
where $b$ is a fixed positive integer, $a_k\in\{0,1,\dots, b-1, b\}$, and there exist  non-negative integers $m_1$ and $m_2$ such that $m_1\ne m_2$ and $a_{m_1}=a_{m_2}$ as $k \to \infty$.

Let us prove \emph{the sufficiency}. Suppose there exist non-negative integers $m_1$ and $m_2$ such that $m_1<m_2$ and 
$$
\left\{q_1q_2\cdots q_{m_1-1}\sigma_{m_1}(x)\right\}=\left\{q_1q_2\cdots q_{m_2-1}\sigma_{m_2}(x)\right\}.
$$

Let us prove the case when 
$$
x\notin \left\{\Delta^Q _{\varepsilon_1\varepsilon_2...\varepsilon_{m_1-1}[\varepsilon_{m_1}-1][q_{m_1+1}-1][q_{m_1+2}-1]...}, \Delta^Q _{\varepsilon_1\varepsilon_2...\varepsilon_{m_2-1}[\varepsilon_{m_2}-1][q_{m_2+1}-1][q_{m_2+2}-1]...}\right\}.
$$
Since $\{x\}=x-[x]$, we have
$$
\left\{q_1q_2\cdots q_{m_1-1}\sigma_{m_1}(x)\right\}=q_1q_2\cdots q_{m_1-1}\sigma_{m_1}(x)-\delta_{m_1-1},
$$
$$
\left\{q_1q_2\cdots q_{m_2-1}\sigma_{m_2}(x)\right\}=q_1q_2\cdots q_{m_1-1}\sigma_{m_1}(x)-\delta_{m_2-1},
$$
and
$$
q_1q_2\cdots q_{m_1-1}\sigma_{m_1}(x)-\delta_{m_1-1}=q_1q_2\cdots q_{m_1-1}\sigma_{m_1}(x)-\delta_{m_2-1}.
$$
Using~\eqref{eq: generalized shift 1}, we obtain
$$
q_1\cdots q_{m_1-1}q_{m_1}x-(q_{m_1}-1)\delta_{m_1-1}-\varepsilon_{m_1}-\delta_{m_1-1}=q_1\cdots q_{m_2-1}q_{m_2}x-(q_{m_2}-1)\delta_{m_2-1}-\varepsilon_{m_2}-\delta_{m_2-1}.
$$
Hence 
$$
x=\frac{q_{m_1}\delta_{m_1-1}-q_{m_2}\delta_{m_2-1}+\varepsilon_{m_1}-\varepsilon_{m_2}}{q_1q_2\cdots q_{m_1} -q_1q_2\cdots q_{m_2}}
$$
is a rational number.

The proof is analogues for the case when 
$$
x\in \left\{\Delta^Q _{\varepsilon_1\varepsilon_2...\varepsilon_{m_1-1}[\varepsilon_{m_1}-1][q_{m_1+1}-1][q_{m_1+2}-1]...}, \Delta^Q _{\varepsilon_1\varepsilon_2...\varepsilon_{m_2-1}[\varepsilon_{m_2}-1][q_{m_2+1}-1][q_{m_2+2}-1]...}\right\}.
$$
\end{proof}

One can note that certain numbers from $[0,1]$ have two different representations by Cantor series \eqref{eq:  series 1}, i.e., 
$$
\Delta^Q _{\varepsilon_1\varepsilon_2\ldots\varepsilon_{m-1}\varepsilon_m000\ldots}=\Delta^Q _{\varepsilon_1\varepsilon_2\ldots\varepsilon_{m-1}[\varepsilon_m-1][q_{m+1}-1][q_{m+2}-1]\ldots}=\sum^{m} _{i=1}{\frac{\varepsilon_i}{q_1q_2\dots q_i}}.
$$
Such numbers are called \emph{$Q$-rational}. The other numbers in $[0,1]$ are called \emph{$Q$-irrational}.

Let $c_1,c_2, \dots , c_m$ be an ordered tuple of integers such that $c_i\in \{0,1, \dots , q_i-1\}$ for all $i=\overline{1,m}$. Then
\emph{a cylinder $\Lambda^{Q} _{c_1c_2...c_m}$ of rank $m$ with base $c_1c_2\ldots c_m$} is a set of the form
$$
\Lambda^{Q} _{c_1c_2...c_m}\equiv\{x: x=\Delta^{Q} _{c_1c_2...c_m\varepsilon_{m+1}\varepsilon_{m+2}\ldots\varepsilon_{m+k}\ldots}\}.
$$

\begin{theorem}  Suppose a number $x$ represented by series~\eqref{eq: series 1} and $x\ne \Delta^Q _{\varepsilon_1...\varepsilon_{m-1}[\varepsilon_m-1][q_{m+1}-1][q_{m+2}-1]...}$ for any $m\in\mathbb N$. Then $x$ is a rational number $\frac a b$ (here $a,b\in\mathbb N$, $a<b$, and $(a,b)=1$) if and only if the condition
$$
\varepsilon_{m+1}=\left[\frac{(q_{m+1}+1)q_1q_2\cdots q_ma-b(q_1q_2\cdots q_m\sigma_{m+1}(x)+q_{m+1}\delta_m)}{b}\right]
$$
holds for any $m\in\mathbb N$, where $\varepsilon_1=\left[\frac{a}{b}q_1\right]$, $[x]$ is the integer part of $x$, and
$$
\delta_m=\varepsilon_1q_2q_3\cdots q_m+\varepsilon_2q_3q_4\cdots q_m+\cdots +\varepsilon_{m-1}q_m+\varepsilon_m.
$$
\end{theorem}
\begin{proof} \emph{Necessity.} 
Let $x$ be a rational number $\frac{a}{b}$. Then for any $m\in\mathbb N$ there exists a cylinder $\Lambda^Q _{\varepsilon_1\varepsilon_2...\varepsilon_m}$ such that $x\in \Lambda^Q _{\varepsilon_1\varepsilon_2...\varepsilon_m}$.
That is,
$$
x\in\left[\frac{\delta_m}{q_1q_2\cdots q_m},\frac{\delta_m+1}{q_1q_2\cdots q_m}\right].
$$

Since $\Delta^Q _{\varepsilon_1...\varepsilon_{m-1}\varepsilon_m[q_{m+1}-1][q_{m+2}-1]...}=\Delta^Q _{\varepsilon_1\varepsilon_2...\varepsilon_{m-1}[\varepsilon_{m}+1]000...}$, where $\varepsilon_m \ne q_m-1$, we do not use representations of the form $\Delta^Q _{\varepsilon_1\varepsilon_2...\varepsilon_m[\varepsilon_m-1][q_{m+1}-1][q_{m+2}-1]...}$ and assume that
$$
\frac{\delta_m}{q_1q_2\cdots q_m}\le x<\frac{\delta_m+1}{q_1q_2\cdots q_m},
$$
$$
{\delta_m}\le q_1q_2\cdots q_m\frac{a}{b}<{\delta_m+1}.
$$

Since 
$$
\sigma_{m+1}(x)=\sum^m _{k=1}{\frac{\varepsilon_k}{q_1q_2\cdots q_k}}+\frac{\sigma^{m+1}(x)}{q_1q_2\cdots q_m},
$$
$\sigma^{m+1}(x)=q_1\cdots q_mq_{m+1}x-\delta_{m+1}$, and $\delta_{m+1}=\varepsilon_{m+1}+q_{m+1}\delta_m$, we have
$$
0\le \frac{a}{b}q_1q_2\cdots q_m -q_1q_2\cdots q_m\sigma_{m+1}(x)+\frac{a}{b}q_1q_2\cdots q_{m+1}-q_{m+1}\delta_m-\varepsilon_{m+1}<1,
$$
$$
\varepsilon_{m+1}\le \frac{a}{b}q_1q_2\cdots q_m(q_{m+1}+1) -q_1q_2\cdots q_m\sigma_{m+1}(x)-q_{m+1}\delta_m<\varepsilon_{m+1}+1.
$$
Hence
$$
\varepsilon_{m+1}=\left[\frac{(q_{m+1}+1)q_1q_2\cdots q_ma-b(q_1q_2\cdots q_m\sigma_{m+1}(x)+q_{m+1}\delta_m)}{b}\right]=[z_{m+1}],
$$
where $\varepsilon_1=\left[\frac{a}{b}q_1\right]$.

\emph{Sufficiency}. If $\varepsilon_{m+1}=[z_{m+1}]$, then
$$
x=\vartheta_{m+1}+\frac{\sigma^{m+1}(x)}{q_1q_2\cdots q_{m+1}}=\frac{\delta_{m+1}}{q_1q_2\cdots q_{m+1}}+\frac{\sigma^{m+1}(x)}{q_1q_2\cdots q_{m+1}}
$$
$$
=\frac{\varepsilon_{m+1}+q_{m+1}\delta_m}{q_1q_2\cdots q_{m+1}}+\frac{\sigma^{m+1}(x)}{q_1q_2\cdots q_{m+1}}=\frac{[z_{m+1}]+q_{m+1}\delta_m}{q_1q_2\cdots q_{m+1}}+\frac{\sigma^{m+1}(x)}{q_1q_2\cdots q_{m+1}}
$$
$$
=\frac{z_{m+1}-\{z_{m+1}\}}{q_1q_2\cdots q_{m+1}}+\frac{q_{m+1}\delta_m}{q_1q_2\cdots q_{m+1}}+\frac{\sigma^{m+1}(x)}{q_1q_2\cdots q_{m+1}}
$$
$$
=\frac{a}{b}{q_1q_2\cdots q_{m+1}}+\sigma^m(x)-\{z_{m+1}\}=\frac{a}{b}
$$
because $\{z_{m+1}\}=\{\varepsilon_{m+1}+\sigma^m(x)\}$, $q_1q_2\cdots q_m\sigma_{m+1}(x)=\delta_m+\sigma^{m+1}(x)$,
and also 
$q_1q_1\cdots q_mx-\delta_m=\sigma^m(x).$
\end{proof}

\end{document}